\documentclass[11pt]{mathscan}    
\usepackage{latexsym, amsmath, amsthm, amssymb, setspace,verbatim}
\usepackage{hyperref}
\usepackage{enumitem}

\title[Diagonal quasi-free automorphisms of Cuntz-Krieger algebras]{On diagonal quasi-free automorphisms of simple Cuntz-Krieger algebras}
\author{Sel\c{c}uk Barlak}
\author{G{\'a}bor Szab{\'o}}

\address{Europa-Universit\"at Flensburg\\
Institut f\"ur mathematische, naturwissenschaftliche und technische Bildung\\
Abteilung f\"ur Mathematik und ihre Didaktik\\
Auf dem Campus 1b\\
DE-24943 Flensburg\\
Germany}
\email{selcuk.barlak@uni-flensburg.de}

\address{Department of Mathematical Sciences\\ 
University of Copenhagen\\
Universitetsparken 5\\
DK-2100 Copenhagen {\O}\\
Denmark}
\email{gabor.szabo@math.ku.dk}

\begin{document}

\newcommand\set[1]{\left\{#1\right\}}  

\newcommand{\IA}[0]{\mathbb{A}} \newcommand{\IB}[0]{\mathbb{B}}
\newcommand{\IC}[0]{\mathbb{C}} \newcommand{\ID}[0]{\mathbb{D}}
\newcommand{\IE}[0]{\mathbb{E}} \newcommand{\IF}[0]{\mathbb{F}}
\newcommand{\IG}[0]{\mathbb{G}} \newcommand{\IH}[0]{\mathbb{H}}
\newcommand{\II}[0]{\mathbb{I}} \renewcommand{\IJ}[0]{\mathbb{J}}
\newcommand{\IK}[0]{\mathbb{K}} \newcommand{\IL}[0]{\mathbb{L}}
\newcommand{\IM}[0]{\mathbb{M}} \newcommand{\IN}[0]{\mathbb{N}}
\newcommand{\IO}[0]{\mathbb{O}} \newcommand{\IP}[0]{\mathbb{P}}
\newcommand{\IQ}[0]{\mathbb{Q}} \newcommand{\IR}[0]{\mathbb{R}}
\newcommand{\IS}[0]{\mathbb{S}} \newcommand{\IT}[0]{\mathbb{T}}
\newcommand{\IU}[0]{\mathbb{U}} \newcommand{\IV}[0]{\mathbb{V}}
\newcommand{\IW}[0]{\mathbb{W}} \newcommand{\IX}[0]{\mathbb{X}}
\newcommand{\IY}[0]{\mathbb{Y}} \newcommand{\IZ}[0]{\mathbb{Z}}

\newcommand{\CA}[0]{\mathcal{A}} \newcommand{\CB}[0]{\mathcal{B}}
\newcommand{\CC}[0]{\mathcal{C}} \newcommand{\CD}[0]{\mathcal{D}}
\newcommand{\CE}[0]{\mathcal{E}} \newcommand{\CF}[0]{\mathcal{F}}
\newcommand{\CG}[0]{\mathcal{G}} \newcommand{\CH}[0]{\mathcal{H}}
\newcommand{\CI}[0]{\mathcal{I}} \newcommand{\CJ}[0]{\mathcal{J}}
\newcommand{\CK}[0]{\mathcal{K}} \newcommand{\CL}[0]{\mathcal{L}}
\newcommand{\CM}[0]{\mathcal{M}} \newcommand{\CN}[0]{\mathcal{N}}
\newcommand{\CO}[0]{\mathcal{O}} \newcommand{\CP}[0]{\mathcal{P}}
\newcommand{\CQ}[0]{\mathcal{Q}} \newcommand{\CR}[0]{\mathcal{R}}
\newcommand{\CS}[0]{\mathcal{S}} \newcommand{\CT}[0]{\mathcal{T}}
\newcommand{\CU}[0]{\mathcal{U}} \newcommand{\CV}[0]{\mathcal{V}}
\newcommand{\CW}[0]{\mathcal{W}} \newcommand{\CX}[0]{\mathcal{X}}
\newcommand{\CY}[0]{\mathcal{Y}} \newcommand{\CZ}[0]{\mathcal{Z}}

\newcommand{\FA}[0]{\mathfrak{A}} \newcommand{\FB}[0]{\mathfrak{B}}
\newcommand{\FC}[0]{\mathfrak{C}} \newcommand{\FD}[0]{\mathfrak{D}}
\newcommand{\FE}[0]{\mathfrak{E}} \newcommand{\FF}[0]{\mathfrak{F}}
\newcommand{\FG}[0]{\mathfrak{G}} \newcommand{\FH}[0]{\mathfrak{H}}
\newcommand{\FI}[0]{\mathfrak{I}} \newcommand{\FJ}[0]{\mathfrak{J}}
\newcommand{\FK}[0]{\mathfrak{K}} \newcommand{\FL}[0]{\mathfrak{L}}
\newcommand{\FM}[0]{\mathfrak{M}} \newcommand{\FN}[0]{\mathfrak{N}}
\newcommand{\FO}[0]{\mathfrak{O}} \newcommand{\FP}[0]{\mathfrak{P}}
\newcommand{\FQ}[0]{\mathfrak{Q}} \newcommand{\FR}[0]{\mathfrak{R}}
\newcommand{\FS}[0]{\mathfrak{S}} \newcommand{\FT}[0]{\mathfrak{T}}
\newcommand{\FU}[0]{\mathfrak{U}} \newcommand{\FV}[0]{\mathfrak{V}}
\newcommand{\FW}[0]{\mathfrak{W}} \newcommand{\FX}[0]{\mathfrak{X}}
\newcommand{\FY}[0]{\mathfrak{Y}} \newcommand{\FZ}[0]{\mathfrak{Z}}

\newcommand{\Fa}[0]{\mathfrak{a}} \newcommand{\Fb}[0]{\mathfrak{b}}
\newcommand{\Fc}[0]{\mathfrak{c}} \newcommand{\Fd}[0]{\mathfrak{d}}
\newcommand{\Fe}[0]{\mathfrak{e}} \newcommand{\Ff}[0]{\mathfrak{f}}
\newcommand{\Fg}[0]{\mathfrak{g}} \newcommand{\Fh}[0]{\mathfrak{h}}
\newcommand{\Fi}[0]{\mathfrak{i}} \newcommand{\Fj}[0]{\mathfrak{j}}
\newcommand{\Fk}[0]{\mathfrak{k}} \newcommand{\Fl}[0]{\mathfrak{l}}
\newcommand{\Fm}[0]{\mathfrak{m}} \newcommand{\Fn}[0]{\mathfrak{n}}
\newcommand{\Fo}[0]{\mathfrak{o}} \newcommand{\Fp}[0]{\mathfrak{p}}
\newcommand{\Fq}[0]{\mathfrak{q}} \newcommand{\Fr}[0]{\mathfrak{r}}
\newcommand{\Fs}[0]{\mathfrak{s}} \newcommand{\Ft}[0]{\mathfrak{t}}
\newcommand{\Fu}[0]{\mathfrak{u}} \newcommand{\Fv}[0]{\mathfrak{v}}
\newcommand{\Fw}[0]{\mathfrak{w}} \newcommand{\Fx}[0]{\mathfrak{x}}
\newcommand{\Fy}[0]{\mathfrak{y}} \newcommand{\Fz}[0]{\mathfrak{z}}

\renewcommand{\phi}[0]{\varphi}
\newcommand{\eps}[0]{\varepsilon}

\newcommand{\id}[0]{\operatorname{id}}		
\newcommand{\ad}[0]{\operatorname{Ad}}
\newcommand{\Aut}[0]{\operatorname{Aut}}
\newcommand{\dst}[0]{\displaystyle}
\newcommand{\cstar}[0]{\ensuremath{\mathrm{C}^*}}
\newcommand{\Ost}[0]{\CO_\infty^{\mathrm{st}}}
\newcommand{\linhull}{\operatorname{span}}

\newtheorem{satz}{Satz}[section]		
\newtheorem{cor}[satz]{Corollary}
\newtheorem{lemma}[satz]{Lemma}
\newtheorem{prop}[satz]{Proposition}
\newtheorem{theorem}[satz]{Theorem}
\newtheorem*{theoremoz}{Theorem}
\newtheorem*{coroz}{Corollary}

\theoremstyle{definition}
\newtheorem{defi}[satz]{Definition}
\newtheorem*{defioz}{Definition}
\newtheorem{defprop}[satz]{Definition \& Proposition}
\newtheorem{nota}[satz]{Notation}
\newtheorem*{notaoz}{Notation}
\newtheorem{rem}[satz]{Remark}
\newtheorem*{remoz}{Remark}
\newtheorem{example}[satz]{Example}
\newtheorem{question}[satz]{Question}
\newtheorem*{questionoz}{Question}

\numberwithin{equation}{section}
\renewcommand{\theequation}{e\thesection.\arabic{equation}}


\begin{abstract}
We show that an outer action of a finite abelian group on a simple Cuntz-Krieger algebra is strongly approximately inner in the sense of Izumi if the action is given by diagonal quasi-free automorphisms and the associated matrix is aperiodic. 
This is achieved by an approximate cohomology vanishing-type argument for the canonical shift restricted to the relative commutant of the set of domain projections of the canonical generating isometries in the fixed point algebra.
\end{abstract}

\maketitle


\section*{Introduction}

Ever since their inception, Cuntz-Krieger algebras \cite{CuntzKrieger1980} and graph \cstar-algebras \cite{KumjianPaskRaeburn1998} form natural classes of nuclear \cstar-algebras satisfying the UCT \cite{RosenbergSchochet1987}, which keep attracting much attention.
One compelling feature is that many of these \cstar-algebras' structural properties have natural characterizations in terms of their matrices and graphs, respectively.
This feature makes them particularly approachable from the point of view of classification, as is evidenced for example by R\o rdam's work \cite{Rordam1995} on simple Cuntz-Krieger algebras. 
More recently, the classification of unital graph \cstar-algebras in purely graph theoretic terms has been completed by Eilers-Restorff-Ruiz-S{\o}rensen; see \cite{EilersRestorffRuizSorensen2016}.

This paper deals with actions of finite abelian groups on Cuntz-Krieger algebras by quasi-free automorphisms in the sense of \cite{Zacharias2000} or \cite{Evans1980}.
This class of automorphisms contains the ones that permute the canonical generating partial isometries, as well as automorphisms for which each of these partial isometries is an eigenvector. 
The latter ones fall into the class of diagonal quasi-free automorphisms, as they fix all range projections of the generating partial isometries.

Every action by diagonal quasi-free automorphisms is conjugate to one that is given by automorphisms for which the canonical partial isometries are eigenvectors; see Proposition~\ref{prop:charac quasi-free diagonal autos}. 
In particular, diagonal quasi-free automorphisms are always homotopic to the identity map. 
Therefore they are approximately inner, provided the Cuntz-Krieger algebra is simple, in which case it must be a unital Kirchberg algebra; see \cite{Rordam1995a, Phillips2000}. 
This is in contrast to automorphisms induced by permutations on the set of the canonical generators. 
In general, such automorphisms act non-trivially on the $K_0$-group of the Cuntz-Krieger algebra, as it is generated by the classes of domain projections of the generating partial isometries; see \cite{Cuntz1981a}.

One might now ask the following question.

\begin{questionoz} \label{ques:intro}
Is every outer action of a finite abelian group by outer diagonal quasi-free automorphisms on a simple, purely infinite Cuntz-Krieger algebra strongly approximately inner in the sense of Izumi \cite{Izumi2004}?
\end{questionoz}

Strong approximate innerness requires that unitaries witnessing approximate innerness can be chosen to lie in the fixed point algebra of the action. If the unitaries can in addition be chosen to give rise to (approximate) group representations, then the action is called approximately representable. 
Although in general different, these notions coincide if the action absorbs the trivial action on a suitable UHF algebra; see \cite{Izumi2004} and \cite{Szabo2017}.

On the one hand, the key feature of these technically looking properties is the duality between approximate representability and the Rokhlin property. 
As illustrated beautifully by Izumi \cite{Izumi2004,Izumi2004a}, this relationship paired with the rigid nature of Rokhlin actions make approximately representable actions particularly accessible to classification via their duals.
For strongly approximately inner actions of cyclic groups of prime power order on $\CO_2$, this perspective was successfully exploited by Izumi, and was even accompanied by a complete range result.

On the other hand, models for actions with the Rokhlin property such as in \cite{BarlakSzabo2017} together with Takai duality \cite{Takai1975} can be used to infer models or structural properties for approximately representable actions.
Along these lines, a characterization of the UCT for the fixed point algebra associated to a $M_{n^\infty}$-absorbing action $\alpha: \IZ_n \curvearrowright A$ on a unital UCT Kirchberg algebra has recently been shown by the first author in joint work with Xin Li; see \cite{BarlakLi2017}. More concretely, $A^\alpha$ satisfies the UCT if and only if there exists an inverse semigroup of $\alpha$-homogeneous partial isometries $\CS \subset A$\footnote{This means that $\CS$ is closed under multiplication and the $*$-operation.} such that $\CS$ generates $A$ as a \cstar-algebra and the projections in $\CS$ generate a Cartan subalgebra of $A$ in the sense of Renault \cite{Renault2008}.

Restricted to Cuntz-Krieger algebras, our question above therefore raises the issue to which extend the latter condition already implies approximate representabilty.
Indeed, the canonical partial isometries of a Cuntz-Krieger algebra generate an inverse semigroup whose projections in turn generate an abelian \cstar-subalgebra.
It is a Cartan subalgebra precisely when the matrix satisfies Condition (I) (see \cite{Renault2008} and \cite{CuntzKrieger1980}), which is equivalent to its spectrum being the Cantor set.

Our main result asserts that the above question indeed has an affirmatively answer if the matrix is aperiodic, that is, if some power of the matrix has only strictly positive entries.

\begin{theoremoz} [see Corollary {\ref{cor:main result}}]
Let $n \geq 2$ be an integer. 
Let $A$ be an aperiodic $n \times n$-matrix with values in $\set{0,1}$.
Let $G$ be a finite abelian group and let $\sigma:G \curvearrowright \CO_A$ be an action by diagonal quasi-free automorphisms.
If $\sigma$ is outer, then it is strongly approximately inner.
\end{theoremoz}

As a consequence, one gets an analogous result for unital graph \cstar-algebras. The conditions we have to impose on the finite graphs are strong connectedness and aperiodicity. Here, a graph is said to be strongly connected aperiodic if there is some $k \geq 1$ such that for any two of its edges $v,w$ there is some path of length $k$ from $v$ to $w$. The notion of (diagonal) quasi-free automorphisms transfers to the setting of unital graph \cstar-algebras in a straightforward manner.

\begin{theoremoz} [see Corollary {\ref{cor:main result graphs}}]
Let $E$ be a finite graph that contains at least two edges.
Let $G$ be a finite abelian group and let $\sigma:G \curvearrowright \cstar(E)$ be an action by diagonal quasi-free automorphisms on the associated graph \cstar-algebra.
Suppose that $E$ is strongly connected aperiodic.
If $\sigma$ is outer, then it is strongly approximately inner. 
\end{theoremoz}

Our main result generalizes Izumi's \cite[Proposition~5.6(2)]{Izumi2004} for actions on the Cuntz algebras by quasi-free automorphisms.
The argument in his work is a variant of R{\o}rdam's proof \cite{Rordam1993} that unital endomorphisms of Cuntz algebras are approximately inner.
A crucial ingredient therein is the Rokhlin property of the one-sided tensorial shift on the CAR algebra; see \cite{BratteliKishimotoRordamStormer1993}.
In the proof, this is used to derive an approximate cohomology vanishing-type result for the canonical endomorphism restricted to the fixed point algebra of the canonical UHF subalgebra. 
Strong approximate innerness of the quasi-free action is then derived from the well-known bijective correspondence between unitaries and unital endomorphisms of Cuntz algebras; see \cite{Cuntz1980}.

Our approach is of a similar nature.
Although less striking than it is for Cuntz algebras, there is still a close connection between certain unitaries and unital endomorphisms of Cuntz-Krieger algebras; see \cite{Zacharias2000} and \cite{ContiHongSzymanski2012}. 
In particular, one can associate to each diagonal quasi-free automorphism a unique unitary so that the automorphism is given on the canonical partial isometries by left multiplication with this unitary. 
The shift map of a Cuntz-Krieger algebra commutes with any given action by diagonal quasi-free automorphisms for which the canonical partial isometries are eigenvectors.
In general, this map may only be a unital completely positive map, but it restricts to a $*$-homomorphism on the relative commutant of the domain projections of the canonical generators inside the fixed point algebra of the action. 
This resulting endomorphism is injective. 
As it is a priori unclear whether this endomorphism has the Rokhlin property (in the sense of \cite{BrownHirshberg2014} or \cite{Rordam1995a}), our approach deviates at this point from Izumi's original approach and instead considers the dilation of the shift map to an automorphism in the sense of \cite{Laca2000}.
Aperiodicity of the matrix enters the game here to make sure that the dilated \cstar-algebra is a unital Kirchberg algebra and the automorphism is aperiodic, provided that the action is outer; see Lemma~\ref{outerness dilated auto} and \cite{Kishimoto1981, Jeong1995, TomsWinter2007}.
A classical theorem of Nakamura \cite{Nakamura2000} then yields that this automorphism has the Rokhlin property.
As the dilated system embeds into the ultrapower of the shift endomorphism, a modified cohomology vanishing-type technique for the unitaries associated with the action can be performed. 
Strong approximate innerness of the involved automorphisms then follows by going back from unitaries to unital endomorphisms.
In fact, a technically more involved argument yields strong approximate innerness for an a priori larger class of actions; see Theorem~\ref{main theorem}.

A consequence of our main result and Izumi's classification result \cite{Izumi2004a} is that outer actions of cyclic groups with prime power order by diagonal quasi-free automorphisms on Cuntz-Krieger algebras isomorphic to $\CO_2$ are classified in terms of their fixed point algebras and some additional information about their dual actions, provided the associated matrices are aperiodic.
In particular this holds for any (possibly non-standard) identification of $\CO_2$ with a Cuntz-Krieger algebra.
Using Kirchberg-Phillips classification \cite{Kirchberg, Phillips2000}, this simplifies as follows in the case of order two automorphisms.

\begin{coroz}[see Corollary {\ref{cor:classification}}]
Let $m, n \geq 2$. 
Let $A$ be an $m \times m$-matrix and $B$ an $n \times n$-matrix with entries in $\set{0,1}$ such that $\CO_A\cong\CO_B\cong\CO_2$ as abstract \cstar-algebras.
Let $\alpha:\IZ_2 \curvearrowright \CO_A$ and $\beta:\IZ_2 \curvearrowright \CO_B$ be two actions by diagonal quasi-free automorphisms.
Suppose that that $A$ and $B$ are aperiodic and that both $\alpha$ and $\beta$ are outer.
Then $\alpha$ and $\beta$ are (cocycle) conjugate if and only if $\CO_A^\alpha$ and $\CO_B^\beta$ are (stably) isomorphic
\end{coroz}

\bigskip
\textbf{Acknowledgement.}
Substantial parts of this work were carried out during a research visit of the second author to the University of Southern Denmark in January 2017 and during a research visit of the first author to the University of Aberdeen in July 2017.
The authors are grateful to both institutions for their hospitality and support.
The first author would like to thank Wojciech Szyma{\'n}ski for inspiring discussions on endomorphisms of Cuntz-Krieger algebras and Rokhlin-type properties for endomorphisms. The authors are also grateful to Joachim Cuntz for valuable comments. The first author was supported by Villum Fonden project grant `Local and global structures of groups and their algebras' (2014-2018). The second author was supported by EPSRC grant EP/N00874X/1, the Danish National Research Foundation through the \emph{Centre for Symmetry and Deformation} (DNRF92), and the European Union's Horizon 2020 research and innovation programme under the Marie Sklodowska-Curie grant agreement 746272.

\section{Preliminaries}

We start by recalling the definition of a Cuntz-Krieger algebra.

\begin{defi}[\cite{CuntzKrieger1980}] \label{defi CK algebras}
Let $n \geq 2$ be an integer. Let $A$ be an $n \times n$ matrix with entries in $\set{0,1}$ and no zero rows or columns. The Cuntz-Krieger algebra $\CO_A$ associated with $A$ is the universal \cstar-algebra generated by a family of partial isometries $\set{s_1,\ldots,s_n}$ subject to the relations
\begin{enumerate}[label=\textup{\arabic*)}, leftmargin=*]
\item $s_i^*s_j = 0$ \quad for $i \neq j$,\label{label:ck:1}
\item $s_i^*s_i = \sum_{j = 1}^n A(i,j) s_js_j^*$ \quad for $i=1,\ldots,n$.\label{label:ck:2}
\end{enumerate}
\end{defi}

\begin{nota} \label{nota: s_mu}
For $i \in \set{1,\ldots, n}$, we set $q_i = s_i^*s_i$ and $p_i = s_is_i^*$. One can check that $\CO_A$ is unital with $1 = \sum_{i = 1}^n p_i$.
If $\mu = (\mu_1,\ldots,\mu_k) \in \set{0,1}^k$ is a multi-index, then we set $s_\mu = s_{\mu_1}\ldots s_{\mu_k} \in \CO_A$. It holds that $s_\mu \neq 0$ if and only if $A(\mu_i,\mu_{i+1}) = 1$ for all $i\in \set{1,\ldots, k-1}$. We denote by $W^k$ the set of all multi-indices $\mu \in \set{0,1}^k$ such that $s_\mu \neq 0$.
\end{nota}

\begin{defi}
An $n \times n$-matrix $A$ with entries in $\set{0,1}$ is said to be aperiodic if there exists some $m \geq 1$ such that $A^m(i,j) > 0$ for all $i,j \in \set{1,\ldots,n}$.
\end{defi}

We shall later need the following well known result for Cuntz-Krieger algebras associated with aperiodic matrices over $\set{0,1}$.

\begin{theorem} \label{thm: Kirchberg aper mat}
Let $n\geq 2$. 
Let $A$ be an aperiodic $n \times n$-matrix with entries in $\set{0,1}$. 
Then $\CO_A$ is a unital Kirchberg algebra.
\end{theorem}
\begin{proof}
Clearly, $\CO_A$ is unital and separable. It follows from \cite[Proposition~1.6]{Cuntz1981} that $\CO_A$ is simple and purely infinite; see also \cite[Theorem~2.14]{CuntzKrieger1980}. 
Furthermore, $\CO_A$ is nuclear, as it is stably isomorphic to a crossed product of an AF algebra by an automorphism; see \cite[Example~2.5]{Rordam1995a}.
\end{proof}

As shown in \cite{Cuntz1980}, there exists a canonical bijective correspondence between unital endomorphisms and unitaries of the Cuntz algebra $\CO_n$ for finite $n$. 
A similar result was obtained in \cite{Zacharias2000} for Cuntz-Krieger-Pimsner algebras. In the case of Cuntz-Krieger algebras, such a correspondence exists for unital endomorphisms that fix all domain projections of the canonical generators and unitaries commuting with these projections; see also \cite{ContiHongSzymanski2012}. 

\begin{prop} \label{prop:correspond graph alg endos}
Let $A$ be an $n \times n$-matrix with entries in $\set{0,1}$ and no zero rows or columns. For any unitary $u \in \CU(\CO_A) \cap \set{q_i \mid 1 \leq i \leq n}'$, there exists a unique unital $*$-endomorphism $\lambda_u$ of $\CO_A$ such that $\lambda_u(q_i) = q_i$ and $\lambda_u(s_i) = us_i$ for all $i \in \set{1,\ldots,n}$.  Conversely, if $\sigma: \CO_A \to \CO_A$ is a unital $*$-endomorphism such that $\sigma(q_i) = q_i$ for all $i$, then $u_\sigma = \sum_{i = 1}^n \sigma(s_i)s_i^*$ is a unitary in $\CU(\CO_A) \cap \set{q_i \mid 1 \leq i \leq n}'$. These assignments are inverse to each other and continuous with respect to the norm and the point-norm topology, respectively.
\end{prop}

\begin{rem} \label{rem:conv formula}
It follows from Proposition~\ref{prop:correspond graph alg endos} that, given two unital endomorphisms  $\rho,\sigma:\CO_A \to \CO_A$ fixing all domain projections of the canonical generators, the associated unitaries satisfy the convolution formula $u_{\sigma \circ \rho} = \sigma(u_\rho)u_\sigma$. In particular, for any $n\geq 1$, one has $\sigma^n = \id$ precisely when $\sigma^{n-1}(u_\sigma)\ldots\sigma(u_\sigma)u_\sigma = 1$.
\end{rem}

\begin{defi}[cf.\ {\cite{Zacharias2000}}] \label{defi:quasi-free}
Let $A$ be an $n \times n$-matrix with entries in $\set{0,1}$ and no zero rows or columns. An automorphism $\sigma \in \Aut(\CO_A)$ is said to be diagonal quasi-free if
\[
\begin{array}{c}
\sigma(q_i)  = q_i \text{ for all } i \in \set{1,\ldots,n} \text{ and}\\
\sigma(\linhull\set{s_i \mid 1 \leq i \leq n}) = \linhull\set{s_i \mid 1 \leq i \leq n}.
\end{array}
\]
\end{defi}

We note that the above definition is slightly more general than the one in \cite{Zacharias2000}. In fact, Zacharias requires a diagonal quasi-free automorphism to restrict to the identity on the set of range projections of the canonical generators.
However, the two notions of diagonal quasi-free automorphisms lead to the same conjugacy classes of automorphisms, as the next result shows.
It is certainly known to experts, and partly contained in \cite{ContiHongSzymanski2012}. For the reader's convenience, we give a proof here.

\begin{prop} \label{prop:charac quasi-free diagonal autos}
Let $A$ be an $n \times n$-matrix with entries in $\set{0,1}$ and no zero rows or columns. 
Consider the finite dimensional \cstar-subalgebra 
\[
B = \linhull\big\{ s_is_j^* \mid 1 \leq i,j\leq n \big\}
\]
of $\CO_A$.
Then the following assertions hold:
\begin{enumerate}[label=\textup{\arabic*)}, leftmargin=*]
\item The bijective map in Proposition~\ref{prop:correspond graph alg endos} restricts to a group isomorphism between the set of diagonal quasi-free automorphisms and $\CU(B) \cap \set{q_i \mid 1 \leq i \leq n}'$. \label{label:cqf:1}
\item Let $G$ be an abelian group. Then every action $\sigma:G \curvearrowright \CO_A$ by diagonal quasi-free automorphisms is conjugate to an action $\rho:G \curvearrowright \CO_A$ with the property that for all $g \in G$ and $i \in \set{1,\ldots,n}$ there exists some $\eta_{g,i} \in \IT$ such that $\rho_g(s_i) = \eta_{g,i} s_i$. \label{label:cqf:2}
\end{enumerate}
\end{prop}
\begin{proof}
\ref{label:cqf:1}: Let $u,w \in \CU(B) \cap \set{q_i \mid 1 \leq i \leq n}'$. Using that $\lambda_u(w) = uwu^*$, one computes for $i \in \set{1,\ldots,n}$ that
\[
\lambda_u \lambda_w(s_i) = \lambda_u(w)us_i = uwu^*us_i = uws_i.
\]
Hence, $\lambda_{uw} = \lambda_u \lambda_w$. By taking $w = u^*$, it follows in particular that $\lambda_u$ is an automorphism. Moreover, if $u = \sum_{i,j =1}^n\eta_{i,j} s_i s_j^*$, then 
\[
\lambda_u(s_j) = \sum_{i =1}^n \eta_{i,j} s_i.
\]
This shows that $\lambda_u$ is diagonal quasi-free. On the other hand, it is clear that $u_\sigma \in \CU(B) \cap \set{q_i \mid 1 \leq i \leq n}'$ if $\sigma \in \Aut(\CO_A)$ is diagonal quasi-free. 

\ref{label:cqf:2}: Let $\sigma:G \curvearrowright \CO_A$ be an action of an abelian group by diagonal quasi-free automorphisms. 
By \ref{label:cqf:1}, we find a unitary representation $u:G \to \CU(B) \cap \set{q_i \mid 1 \leq i \leq n}'$ such that $\sigma_g = \lambda_{u(g)}$ for all $g \in G$. 
As $G$ is abelian, $u$ maps into some MASA of the finite dimensional \cstar-algebra $B \cap \set{q_i \mid 1 \leq i \leq n}'$. 
Note that $\linhull\set{p_i \mid 1\leq i \leq n}$ is a MASA in this \cstar-algebra. 
Using that every two MASAs in a finite dimensional \cstar-algebra are unitarily conjugate, we find some $w \in \CU(B) \cap \set{q_i \mid 1 \leq i \leq n}'$ such that
\[
wu(g)w^* \in \linhull\set{p_i \mid 1\leq i \leq n} \quad \text{for all } g \in G.
\]
It follows from \ref{label:cqf:1} that $\lambda_{wu(g)w^*} = \lambda_w \sigma_g \lambda_w^{-1}$ for all $g \in G$. Hence, the action $\rho:G \curvearrowright \CO_A$ given by $\rho_g = \lambda_{wu(g)w^*}$ is conjugate to $\sigma$. Furthermore, it is straighthforward to check that for $g \in G$ and $i \in \set{1,\ldots,n}$ there exists $\eta_{g,i} \in \IT$ such that $\rho_g(s_i) = \eta_{g,i} s_i$. This finishes the proof.
\end{proof}

\section{Main results}

\begin{nota}
We call an automorphism $\alpha$ on a unital \cstar-algebra $A$ aperiodic if $\alpha^k$ is outer for all $k \neq 0$.

For a free filter $\omega$ on $\IN$, we denote the $(\omega)$-sequence algebra of $A$ by
\[
A_\omega = \ell^\infty(\IN,A) \big / \set{(x_n)_n \in \ell^\infty(\IN,A) \mid  \lim_{n \to \omega} \|x_n\| = 0}.
\]
If $\omega_\infty$ is the Fr{\'e}chet filter, we simply write $A_\infty = A_{\omega_\infty}$. There is a canonical embedding of $A$ into $A_\omega$ by (representatives of) constant sequences. Any endomorphism $\psi:A \to A$ induces endomorphisms $\psi_\omega:A_\omega \to A_\omega$ and $\psi_\infty:A_\infty \to A_\infty$. These assignments are functorial, so that any discrete group action $\alpha:G \curvearrowright A$ induces group actions $\alpha_\infty$ and $\alpha_\omega$ of $G$ on $A_\infty$ and $A_\omega$, respectively.
\end{nota}

\begin{prop} \label{outerness dilated auto}
Let $A$ be a separable, unital \cstar-algebra and $\psi: A \to A$ be a unital and injective $*$-homomorphism. Let
\[
(B, \bar{\psi}) = \lim_{\longrightarrow} \set{(A,\psi),\psi}
\]
denote its dilation to a $*$-automorphism; see \cite{Laca2000}. If $\psi$ is not an asymptotically inner $*$-automorphism\footnote{This means $\psi$ is a $*$-automorphism and there is no norm-continuous path $(u_t)_{t \in [0,\infty)}$ of unitaries in $A$ such that $\psi = \lim\limits_{t \to \infty} \ad(u_t)$.}, then $\bar{\psi}$ is outer. 
In particular, for every $k \geq 1$ one has that if $\psi^k$ is not an asymptotically inner $*$-automorphism, then ${\bar{\psi}}^k$ is outer.
\end{prop}
\begin{proof}
By the standard construction of the inductive limit, $B$ arises as the closure of the $*$-algebra
\[
\set{[(a,\psi(a),\psi^2(a),\ldots)_{n \geq \ell}] \in A_\infty \mid a \in A,\ \ell \geq 1} \ \subset \ A_\infty,
\]
and on this dense subset the $*$-endomorphism $\bar{\psi}$ coincides with $\psi_\infty$.
Here the notation indicates that we only consider the tail of a representing sequence, which makes sense by definition of $A_\infty$. 

Suppose $\bar{\psi}=\ad(v)$ for some unitary $v\in B$. By the definition of $B$, we can find a unitary\footnote{Note that the dense $*$-subalgebra of $B$ is closed under functional calculus.} $u \in A$ and $\ell \geq 1$ such that $\|v-\bar{u}\|<1/2$, where 
\[
\bar{u} = [(u,\psi(u),\psi^2(u),\ldots)_{n \geq \ell}] \in B.
\]
Without loss of generality, let us assume $\ell=1$ here.
Let $a \in A$ with $\|a\|\leq 1$ and write $\bar{a} = [(\psi^{n-1}(a))_{n \geq 1}] \in A_\infty$.
We use that $\psi$ is injective and calculate
\[
\renewcommand\arraystretch{1.25}
\begin{array}{lcl}
\| \psi(a) - uau^* \| & = & \limsup\limits_{n \to \infty} \| \psi^n(\psi(a) - uau^*) \| \\
 & = & \limsup\limits_{n \to \infty} \| \psi(\psi^n(a)) - \psi^n(u) \psi^n(a) \psi^n(u)^* \| \\
 & = & \|\bar{\psi}(\bar{a})-\ad(\bar{u})(\bar{a})\| \\
 & \leq & 2\|v-\bar{u}\| \ < \ 1.
 \end{array}
\]
Hence, $\ad(u^*) \circ \psi$ has distance less than one in the operator norm to the identity operator on $A$. It thus follows from \cite[Theorem~8.7.7]{Pedersen1979} and \cite[Lemma~2.14 with proof]{DadarlatEilers2001} that $\psi$ is an asymptotically inner $*$-automorphism.

If $k \geq 1$ and ${\bar{\psi}}^k$ is inner, it now follows from the canonical isomorphism
\[
(B,{\bar{\psi}}^k) \cong \lim_{\longrightarrow} \set{(A,\psi^k),\psi^k}
\]
that $\psi^k$ is an asymptotically inner $*$-automorphism.
\end{proof}

\begin{rem} \label{rem:shift}
Let $A$ be an $n \times n$-matrix with entries in $\set{0,1}$ and no zero rows or columns. There exists a canonical unital completely positive map $\phi:\CO_A \to \CO_A$  given by $\phi(x) = \sum_{i =1}^n s_i x s_i^*$. Furthermore, it restricts to a $*$-endomorphism of $\CO_A \cap \set{q_i \mid 1 \leq i \leq n}'$. It is injective, as one finds for any $0\neq x \in \CO_A \cap \set{q_i \mid 1 \leq i \leq n}'$ some $j \in \set{1,\ldots,n}$ such that 
\[
0 \neq xq_j = q_jxq_j,
\]
and consequently $0 \neq s_jx s_j^* = p_j \phi(x)$.
\end{rem}

\begin{lemma} \label{simplicity dilated system}
Let $n \geq 2$. 
Let $A$ be an aperiodic $n \times n$-matrix with entries in $\set{0,1}$ and no zero rows or columns. 
Let $C \subset \mathrm \CO_A$ be a simple \cstar-subalgebra containing $\set{q_i \mid 1\leq i \leq n}$.
Suppose that $\phi(C) \subset C$ and that the restricted $*$-endomorphism
\[
\phi:C \cap \set{q_i \mid 1\leq i \leq n}' \to C \cap \set{q_i \mid 1\leq i \leq n}'
\]
is not surjective. Denote its automorphic dilation by
\[
(B, \bar{\phi}) = \lim_{\longrightarrow} \set{(C \cap \set{q_i \mid 1\leq i \leq n}',\phi),\phi}.
\]
Then $B$ is a unital simple \cstar-algebra and $\bar{\phi}$ an aperiodic automorphism.
\end{lemma}
\begin{proof}
Observe that $\cstar(\set{q_i \mid 1\leq i \leq n}) \subset C$ is a commutative \cstar-sub\-algebra containing the unit of $\CO_A$. Hence, $C$ is unital and by Remark~\ref{rem:shift} it follows that
\[
\phi:C \cap \set{q_i \mid 1\leq i \leq n}' \to C \cap \set{q_i \mid 1\leq i \leq n}'
\]
is unital and injective, yielding that $B$ is unital as well. Moreover, aperiodicity of $\bar{\phi}$ follows immediately from Proposition~\ref{outerness dilated auto}.

It remains to show that $B$ is simple. Using Notation~\ref{nota: s_mu}, one computes for $k \geq 1$ and $x \in \CO_A$, that 
\[
\phi^k(x) = \sum\limits_{\nu \in W^k} s_\nu x s_\nu^*.  
\]
Let $\set{r_i \mid 1 \leq i \leq m}$ be the set of minimal projections in $\cstar(\set{q_i \mid 1\leq i \leq n})$. One has that
\[
C \cap \set{q_i \mid 1\leq i \leq n}' = \bigoplus_{i = 1}^m r_iCr_i.
\]
Now let $x \in C \cap \set{q_i \mid 1\leq i \leq n}'$ be a non-zero element.
Find $i \in \set{1,\ldots,m}$ such that $0 \neq xr_i = r_ixr_i$. 
For any $k \geq 1$ and $j \in \set{1,\ldots,m}$ we have that
\begin{equation} \label{eq:shift^k}
r_j \phi^k(r_ixr_i)r_j = \sum_{ \nu \in W^k }  r_j s_\nu r_i x r_i s_\nu^* r_j 
= \sum_{ \nu \in W_{j,i}^k }  s_\nu x s_\nu^*,
\end{equation}
where
\[
W_{j,i}^k = \set{\nu \in W^k \mid p_{\nu_1} \leq r_j,\ r_i \leq q_{\nu_k}}.
\]
Find $s,t \in \set{1,\ldots,n}$ such that $p_s \leq r_j$ and $r_i \leq q_t$.
As $A$ is aperiodic, there exists some $m_0 \geq 1$ such that $A^m(u,v) > 0$ for all $u,v \in \set{1,\ldots,n}$ and $m \geq m_0$.
In particular, we find for each $m \geq m_0$ some $\mu \in W^m$ with $\mu_1 = s$ and $\mu_m = t$.
It now follows from \eqref{eq:shift^k} that $0 \neq r_j\phi^m(r_ixr_i)r_j \in r_jCr_j $ for $m \geq m_0$, as
\[
s_\mu(r_j \phi^m(r_ixr_i)r_j)s_\mu^* = \sum_{\nu \in W_{j,i}^m }  s^*_\mu s_\nu x s_\nu^*s_\mu = s_\mu^* s_\mu x s_\mu^* s_\mu = q_t x q_t \neq 0.
\]
As $C$ is simple, $r_j\phi^m(r_ixr_i)r_j \in r_jCr_j$ is a full element for all $j\in \set{1,\ldots,n}$ and $m \geq m_0$. 
Hence, $\phi^m(x)$ is full in $C \cap \set{q_i \mid 1\leq i \leq n}'$ for any non-zero $x \in C \cap \set{q_i \mid 1\leq i \leq n}'$ and $m \geq m_0$. 
We conclude that $B$ is simple. This finishes the proof. 
\end{proof}

Next comes the main technical result of this paper. Before we state it, we need the following definition due to Izumi.

\begin{defi}[see {\cite[Definition~3.6]{Izumi2004}}]
Let $G$ be a finite abelian group. An action $\alpha: G \curvearrowright A$ on a separable, unital \cstar-algebra is said to be strongly approximately inner if for each $g \in G$ there exists a sequence of unitaries $\set{u_{g,n} \mid n \in \IN} \subset \CU(A^\alpha)$ such that $\alpha_g = \lim_{n \to \infty} \ad(u_{g,n})$.
\end{defi}

\begin{theorem} \label{main theorem}
Let $n \geq 2$.
Let $A$ be an aperiodic $n \times n$-matrix with values in $\set{0,1}$.
Let $G$ be a finite abelian group and let $\sigma: G \curvearrowright \CO_A$ be an action such that $\sigma_g(q_i) = q_i$ for all $g \in G$ and $i \in \set{1,\ldots,n}$.
For $g \in G$, denote by $u_{\sigma_g} \in \CO_A \cap \set{q_i \mid 1 \leq i \leq n}'$ the unitary associated with $\sigma_g$.
Assume furthermore that 
\begin{enumerate}[label=\textup{\arabic*)}]
\item $\phi^k(u_{\sigma_g}) \in \CO_A^\sigma$ for all $g \in G$ and $k \geq 0$;
\item there exists a simple, nuclear \cstar-subalgebra $C \subset \CO_A^\sigma$ containing the set of domain projections $\set{q_i \mid 1 \leq i \leq n}$ such that $\phi(C) \subset C$ and the restricted $*$-endomorphism
\[
\phi:C \cap \set{q_i \mid 1 \leq i \leq n}' \to C \cap \set{q_i \mid 1 \leq i \leq n}'
\]
is not surjective.
\end{enumerate}
If $\sigma$ is outer, then it is strongly approximately inner.
\end{theorem}
\begin{proof}
Let $g \in G$. As $u_{\sigma_g}$ is fixed by $\sigma_g$ and $\sigma_g^n = \id$ for some $n \geq 1$, it follows from Remark~\ref{rem:conv formula} that
\[
u_{\sigma_g}^n = \sigma_g^{n-1}(u_{\sigma_g})\ldots\sigma_g(u_{\sigma_g})u_{\sigma_g} = 1.
\]
For $k \geq 0$, define unitaries $u_{g,k} \in \CO_A^\sigma \cap \set{q_i \mid 1 \leq i \leq n}'$ by
\[
u_{g,0} = 1,\ u_{g,1} = u_{\sigma_g} \text{ and } u_{g,k+1} = u_{g,\sigma}\phi(u_{g,\sigma}) \ldots \phi^k(u_{g,\sigma}).
\]
From the equality $\sigma_g \circ \phi = \ad(u_{\sigma_g}) \circ \phi \circ \sigma_g$, it follows that 
\[
\sigma_g \circ \phi^k = \ad(u_{g,k}) \circ \phi^k \circ \sigma_g \quad \text{for all } k \geq 0.
\]
If $x \in \CO_A^{\sigma_g}$ and $\phi^k(x)$ are both fixed by $\sigma_g$, then this yields $[\phi^k(x),u_{g,k}] = 0$.
Using $\phi(C) \subset C\subset\CO_A^\sigma$ this shows that for $x \in C$ and $0 \leq i \leq k$,
\[
[\phi^k(x),u_{g,i}] = [\phi^i(\phi^{k-i}(x)),u_{g,i}] = 0.
\]
Thus, as $u_{g,i}^*u_{g,i + j} = \phi^i(u_{g,j})$ for all $i,j \in \IN$, it follows that
\[
\phi^k(x) \in C \cap \set{\phi^i(u_{g,j}),\ q_l \mid g \in G,\ 0 \leq j \leq k,\ 0 \leq i \leq k - j,\ 1 \leq l \leq n}'
\]
for all $x \in C \cap \set{q_i \mid 1 \leq i \leq n}'$ and $k \geq 1$.
In particular, $[\phi^k(u_{\sigma_g}),\phi^l(u_{\sigma_g})]= 0$ for all $k,l \geq 0$.
Hence, $u_{g,k}^n = 1$ for all $k \geq 0$ as $u_{\sigma_g}^n = 1$. Thus, each scalar in the spectrum of $u_{g,k}$ is an $n$-th root of unity.

Consider the automorphic dilation of $\phi$ on $C \cap \set{q_i \mid 1 \leq i \leq n}'$
\[
(B, \bar{\phi}) = \lim_{\longrightarrow} \set{(C \cap \set{q_i \mid 1 \leq i \leq n}',\phi),\phi}.
\]
By Lemma \ref{simplicity dilated system}, $B$ is a unital, simple, nuclear \cstar-algebra and $\bar{\phi}$ an aperiodic automorphism on $B$. 
By the standard construction of the inductive limit, there exists a unital, injective and equivariant $*$-homomorphism
\[
\tilde{\rho}:(B,\bar{\phi}) \to \big( (C \cap \set{q_i \mid 1 \leq i \leq n}')_\infty,\phi_\infty \big)
\]
with the image being the closure of 
\[
\set{[(x,\phi(x),\phi^2(x),\ldots)_{n \geq k}] \in C_\infty \mid x \in C \cap \set{q_i \mid 1 \leq i \leq n}',\ k\geq 1}.
\]
Observe that
\[
\tilde{\rho}(B) \subset (\CO_A^\sigma)_\infty \cap \set{\phi^i(u_{g,j}),\ q_k \mid g \in G,\ i,j \geq 0,\ 1\leq k \leq n}'.
\]

Let $\omega$ be a free ultrafilter on $\IN$. 
As $\CO_A$ is a unital Kirchberg algebra by Theorem~\ref{thm: Kirchberg aper mat}, it follows from \cite[Proposition~3.4]{KirchbergPhillips2000} that $(\CO_A)_\omega \cap \CO_A'$ is unital, simple and purely infinite. 
As $\sigma$ is a pointwise outer action, so is $\sigma_\omega$; see \cite[proof of Lemma~2]{Nakamura2000}. Hence, the fixed point algebra $((\CO_A)_\omega \cap \CO_A')^{\sigma_\omega}$ is unital, simple and purely infinite by \cite[Theorem 3.1]{Kishimoto1981} and \cite[Theorem~3]{Jeong1995}, and therefore admits a unital embedding of $\CO_\infty$. Note that $\phi_\omega$ acts trivially on $(\CO_A)_\omega \cap \CO_A'$. 
A reindexation argument now shows the existence of a unital embedding 
\[
\iota:\CO_\infty \to \big( (\CO_A)_\omega \cap (\rho(B)\cup \CO_A)' \big)^{\sigma_\omega},
\]
where $\rho(B) \subset (\CO_A)_\omega$ denotes the image of $\tilde{\rho}(B)$ under the canonical surjection $(\CO_A)_\infty \to (\CO_A)_\omega$.
Set $D = \mathrm C^*(\rho(B),\iota(\CO_\infty))$, which is a $\phi_\omega$-invariant \cstar-subalgebra of 
\[
(\CO_A^\sigma)_\omega \cap \set{\phi^i(u_{g,j}),\ q_k \mid  g\in G,\ i,j \geq 0,\ 1 \leq k \leq n}'
\]
containing the unit of $(\CO_A)_\omega$. In fact, the assignment $b\otimes c\mapsto \rho(b)\cdot\iota(c)$ yields an equivariant $*$-isomorphism
\[
(B \otimes \CO_\infty, \bar{\phi} \otimes \id) \to (D,\phi_\omega).
\]

Now $B \otimes \CO_\infty$ is a unital Kirchberg algebra and $\bar{\phi} \otimes \id$ is an aperiodic automorphism, so that $\bar{\phi} \otimes \id$ has the Rokhlin property by \cite[Theorem~1]{Nakamura2000}.
For given $r \geq 1$, it is thus possible to find projections $e_0,f_0,\ldots,e_{r-1},f_{r-1},f_r \in D$ such that
\begin{equation}\label{eq:Rok:1}
\sum\limits_{i=1}^{r-1}e_i + \sum\limits_{j=1}^r f_j = 1;
\end{equation}
\begin{equation} \label{eq:Rok:2}
\begin{split}
\phi_\omega(e_i) \approx e_{i+1} \text{ and } \phi_\omega(f_j) \approx f_{j+1} \text{ for all } i=0,\ldots,r-1 \\ \text{ and } j=1,\ldots,r,
\text{ where we set  }e_r = e_0 \text{ and } f_{r+1} = f_0.
\end{split}
\end{equation}
By a reindexation trick, we may actually find such elements in
\[
(\CO_A^\sigma)_\omega \cap \set{\phi^i(u_{g,j}),\ q_k \mid g \in G,\ i,j \geq 0,\ 1 \leq k \leq n}'
\]
satisfying the second property \eqref{eq:Rok:2} on the nose.

Now fix $g \in G$. As $u_{g,r}$ and $u_{g,r+1}$ have finite spectrum, we find continuous maps
\[
z_g^{(i)}:[0,1] \to \CU\Big( (\CO_A^\sigma)_\omega \cap \set{e_0,f_0,\ldots,e_{r-1},f_{r-1},f_r,q_j \mid 1 \leq j \leq n}' \Big)
\]
for $i=0,1$ such that
\begin{equation} \label{eq:longPath:1}
z_g^{(0)}(0) = u_{g,r},\ z_g^{(1)}(0) = u_{g,r+1} \text{ and } z_g^{(0)}(1) = z_g^{(1)}(1) = 1;
\end{equation}
\begin{equation} \label{eq:longPath:2}
\begin{split}
\phi_\omega^k(z_g^{(i)}(s)) \in (\CO_A^\sigma)_\omega \cap \set{e_0,f_0,\ldots,e_{r-1},f_{r-1},f_r,q_j \mid 1 \leq j \leq n}' \\
\text{ for all } s \in [0,1],\  i=0,1 \text{ and } k \geq 0;
\end{split}
\end{equation}
\begin{equation} \label{eq:longPath:3}
\| z_g^{(i)}(s) - z_g^{(i)}(t) \| \leq 2\pi|s - t| \text{ for all } s,t \in [0,1] \text{ and } i=0,1.
\end{equation}
Define unitaries 
\[
z^{(i)}_{g,j_i} = z_g^{(i)}(j_i/(r+i)) \quad \text{for }0 \leq j_i \leq r - 1 + i,\ i = 0,1. 
\]
Then it follows from \eqref{eq:longPath:3} that 
\begin{equation} \label{eq:longPath:4}
\| z^{(i)}_{g,j_i} - z^{(i)}_{g,j_i + 1} \| \leq 2\pi/r \text{ for all } 0 \leq j_i \leq r - 2 + i \text{ and } i = 0,1.
\end{equation}
Define a unitary $z_g \in \CU\Big( (\CO_A^\sigma)_\omega \cap \set{q_i \mid 1 \leq i \leq n}' \Big)$ by
\[
z_g = \sum_{k = 0}^{r - 1} u_{g,k} \phi_\omega^k (z^{(0)}_{g,k}) e_k + \sum_{l = 0}^r u_{g,l} \phi_\omega^l (z^{(1)}_{g,l}) f_l.
\]
One computes that
\[
\renewcommand\arraystretch{1.5}
\begin{array}{cl}
\multicolumn{2}{l}{ z_g\phi_\omega (z_g)^* } \\
\stackrel{\eqref{eq:Rok:1},\eqref{eq:Rok:2},\eqref{eq:longPath:1}}{=} & u_{g,r} e_0\phi_\omega(e_{r-1})\phi(u_{g,r-1})^* + u_{g,r+1} f_0\phi_\omega(f_r)\phi(u_{g,r})^* \\
& \dst + \sum_{k=1}^{r-1} u_{g,k} \phi_\omega^k (z^{(0)}_{g,k}) e_k \phi_\omega(e_{k-1}) \phi_\omega^k(z^{(0)}_{g,k-1})^*\phi(u_{g,k-1})^* \\
& \dst + \sum_{l=1}^{r} u_{g,l} \phi_\omega^l (z^{(1)}_{g,l}) f_l \phi_\omega(f_{l-1}) \phi_\omega^l(z^{(1)}_{g,l-1})^*\phi(u_{g,l-1})^* \\
\stackrel{\eqref{eq:Rok:2},\eqref{eq:longPath:2}}{=} & u_{g,r} \phi(u_{r-1})^* e_0 + u_{g,r+1} \phi(u_{g,r})^* f_0 \\
 & \dst + \sum_{k=1}^{r-1} u_{g,k} \phi_\omega^k (z^{(0)}_{g,k}) \phi_\omega^k(z^{(0)}_{g,k-1})^*\phi(u_{g,k-1})^* e_k \\
 & \dst + \sum_{l=1}^{r} u_{g,l} \phi_\omega^l (z^{(1)}_{g,l}) \phi_\omega^l(z^{(1)}_{g,l-1})^*\phi(u_{g,l-1})^* f_l \\
\stackrel{\eqref{eq:longPath:4}}{=}_{\makebox[0pt]{\footnotesize \hspace{2mm} $2\pi/r$}} & u_{g,r}\phi(u_{g,r-1})^*e_0 + u_{g,r+1}\phi(u_{g,r})^*f_0 \\
 & \dst + \sum_{k=1}^{r-1} u_{g,k} \phi(u_{g,k-1})^*e_k \\
 & \dst + \sum_{l=1}^{r} u_{g,l} \phi(u_{g,l-1})^*f_l \\
= & \dst \sum_{k=0}^{r-1} u_{\sigma_g} e_k + \sum_{l=0}^r u_{\sigma_g} f_l \\
\stackrel{\eqref{eq:Rok:1}}{=} & u_{\sigma_g}.
\end{array}
\]
Now let $\eps>0$ be given.
By choosing $r \in \IN$ so that $2\pi/r < \eps$, we may represent $z_g$ via unitaries and thus find a unitary 
\[
w_g \in \CO_A^\sigma \cap \set{q_i \mid 1 \leq i \leq n}' \quad\text{with}\quad \|u_{\sigma_g} - w_g\phi(w_g)^* \| \leq \eps.
\]
It is easy to check that $\ad(w_g) = \lambda_{w_g\phi(w_g)^*}$, so that
\[
\|\sigma(s_i) - \ad(w_g)(s_i)\| \leq \eps \quad \text{for all } i \in \set{1,\ldots,n}.
\]
This yields a sequence of unitaries $\set{w_{g,n}}_{n\in\IN} \subset \CO_A^\sigma \cap \set{q_i \mid 1 \leq i \leq n}'$ such that $\sigma_g = \lim_{n \to \infty} \ad(w_{g,n})$.
The proof is complete.
\end{proof}

\begin{cor} \label{cor:main result}
Let $n \geq 2$. 
Let $A$ be an aperiodic $n \times n$-matrix with values in $\set{0,1}$.
Let $G$ be a finite abelian group and $\sigma:G \curvearrowright \CO_A$ an action by diagonal quasi-free automorphisms.
If $\sigma$ is outer, then it is strongly approximately inner.
\end{cor}
\begin{proof}
By Proposition~\ref{prop:charac quasi-free diagonal autos}\ref{label:cqf:2}, we may assume that for each $g \in G$ and $i \in \set{1,\ldots,n}$, there exists some $\eta_{g,i} \in \IT$ such that $\sigma_g(s_i) = \eta_{g,i} s_i$.
Then
\[
u_{\sigma_g} = \sum_{i = 1}^n \eta_{g,i} p_i \ \in \ \CO_A^\sigma \cap \set{q_i \mid  1 \leq i \leq n}'.
\]
One checks that for each $g \in G$, $\sigma_g \circ \phi = \phi \circ \sigma_g$ so that $\phi(\CO_A^\sigma) \subset \CO_A^\sigma$. As $\sigma$ acts by diagonal quasi-free automorphisms, $\set{q_i \mid 1 \leq i \leq n}$ is contained in $\CO_A^\sigma$. Furthermore, as $\sigma$ is outer, $\CO_A^\sigma$ is a unital Kirchberg algebra. 
The assumptions of Theorem~\ref{main theorem} are thus satisfied and prove the claim, provided that 
\[
\phi: \CO_A^\sigma \cap \set{q_i \mid 1 \leq i \leq n}' \to \CO_A^\sigma \cap \set{q_i \mid 1 \leq i \leq n}'.
\]
is not surjective.

Suppose by way of contradiction that $\phi$ restricts to a surjective endomorphism of $\CO_A^\sigma \cap \set{q_i \mid 1 \leq i \leq n}'$. For $i \in \set{1,\ldots,n}$, find $x_i \in \CO_A^\sigma \cap \set{q_i \mid 1 \leq i \leq n}'$ such that $\phi(x_i) = p_i$. One computes that
\[
x_iq_i = q_ix_iq_i = s_i^* \phi(x_i) s_i = s_i^* p_i s_i = q_i. 
\]
A similar computation shows that  $x_i q_j = 0$ if $j \neq i$. 
Hence, $x_i = q_i$ and it follows that 
\[
\phi\left(\sum_{i=1}^n q_i \right) = \sum_{i = 1}^n p_i = 1.
\]
As $\phi$ is injective on $\CO_A^\sigma \cap \set{q_i \mid 1 \leq i \leq n}'$, we conclude that $\sum_{i=1}^n q_i = 1$. This in turn implies that $A$ is a permutation matrix, which contradicts the assumption that $A$ is aperiodic. The proof is complete.
\end{proof}

Our main result also applies to unital graph \cstar-algebras; see \cite{KumjianPaskRaeburn1998} for a definition. The conditions we have to impose on the finite graphs are strong connectedness and aperiodicity. Here, a graph is said to be strongly connected aperiodic if there is some $k \geq 1$ such that for any two of its edges $v,w$ there is some path of length $k$ from $v$ to $w$. The notion of quasi-free automorphisms transfers to the setting of unital graph \cstar-algebras in a straightforward manner. In particular, we call an automorphism of a unital graph \cstar-algebra diagonal quasi-free if it preserves the span of the canonical generating partial isometries and fixes all vertex projections.

\begin{cor} \label{cor:main result graphs}
Let $E$ be a finite graph that contains at least two edges.
Let $G$ be a finite abelian group and $\sigma:G \curvearrowright \cstar(E)$ be an action by diagonal quasi-free automorphisms on the associated graph \cstar-algebra.
Suppose that $E$ is strongly connected aperiodic.
If $\sigma$ is outer, then it is strongly approximately inner. 
\end{cor}
\begin{proof}
Denote by $E^1$ the set of edges of $E$. The edge matrix of $E$ is the $E^1 \times E^1$ matrix $A_E$ defined by
\[
A_E(e,f) = 
\begin{cases}
1 &,\ \text{if } r(e) = s(f), \\
0 &, \ \text{otherwise}.
\end{cases}
\]
The canonical partial isometries $\set{s_e \mid e \in E^1} \subset \cstar(E)$ define a Cuntz-Krieger family for $A_E$; see \cite[Section~1]{KumjianPaskRaeburn1998}. It follows that $\cstar(E)$ and $\CO_{A_E}$ are canonically isomorphic. In particular, this isomorphism intertwines (diagonal) quasi-free automorphisms. By Corollary~\ref{cor:main result}, it is thus enough to check that $A_E$ is aperiodic. However, this holds true as $A_E$ is the adjacency matrix of a graph, which in turn is strongly connected aperiodic since $E$ has this property.
\end{proof}

In combination with Izumi's classification theorem \cite[Theorem~4.6]{Izumi2004}, we obtain from Corollary~\ref{cor:main result} that outer actions of cyclic groups with prime power order by diagonal quasi-free automorphisms on Cuntz-Krieger algebras isomorphic to $\CO_2$ are classifiable in terms of their fixed point algebras and some additional information about their dual actions, provided the associated matrices are aperiodic. In the case of $\IZ_2$-actions, the latter information is not needed. We thus derive the following result.

\begin{cor} \label{cor:classification}
Let $m, n \geq 2$. 
Let $A$ be an $m \times m$-matrix and $B$ be an $n \times n$-matrix with entries in $\set{0,1}$ such that $\CO_A\cong\CO_B\cong\CO_2$ as abstract \cstar-algebras.
Let $\alpha:\IZ_2 \curvearrowright \CO_A$ and $\beta:\IZ_2 \curvearrowright \CO_B$ be two actions by diagonal quasi-free automorphisms.
Suppose that $A$ and $B$ are aperiodic and that both $\alpha$ and $\beta$ are outer.
Then $\alpha$ and $\beta$ are (cocycle) conjugate if and only if $\CO_A^\alpha$ and $\CO_B^\beta$ are (stably) isomorphic
\end{cor}
\begin{proof}
The claim follows by combining Corollary~\ref{cor:main result}, \cite[Theorem~4.8]{Izumi2004} and Kirchberg-Phillips classification \cite{Kirchberg,Phillips2000}.
\end{proof}

\bibliographystyle{plain}
\bibliography{quasifree}
\end{document}